\documentclass{article}

\usepackage[utf8]{inputenc}
\usepackage[T1]{fontenc}
\usepackage[margin=1in]{geometry}
\usepackage{amsmath,amssymb,amsthm,stmaryrd,amsfonts}
\usepackage{graphicx}
\usepackage{tikz}
\usepackage{tikz-cd}
\usepackage{mathtools}
\usepackage{hyperref}
\usepackage[capitalize]{cleveref}
\usepackage{enumitem}
\usepackage{microtype}
\usepackage{mathrsfs}
\usepackage{bbm}
\usepackage{bm}
\usepackage{booktabs}
\usepackage{array}
\usepackage{subcaption}
\usepackage{xcolor}

\newcommand{\ord}{\mathrm{ord}}
\newcommand{\tri}{\mathrm{tri}}

\DeclareMathOperator{\Gal}{Gal}
\DeclareMathOperator{\Spf}{Spf}
\DeclareMathOperator{\rig}{rig}
\DeclareMathOperator{\Fil}{Fil}
\DeclareMathOperator{\gr}{gr}
\DeclareMathOperator{\Hom}{Hom}
\DeclareMathOperator{\GL}{GL}
\DeclareMathOperator{\GSp}{GSp}

\newcommand{\Q}{\mathbb{Q}}
\newcommand{\Z}{\mathbb{Z}}
\newcommand{\Fp}{\mathbb{F}_p}
\newcommand{\Qp}{\mathbb{Q}_p}
\newcommand{\Zp}{\mathbb{Z}_p}
\newcommand{\R}{\mathbb{R}}
\newcommand{\F}{\mathbb{F}}

\usepackage{amscd}

\newcommand{\cris}{\mathrm{cris}}
\newcommand{\dR}{\mathrm{dR}}

\newcommand{\cO}{\mathcal{O}}
\newcommand{\CNL}{\mathrm{CNL}}
\DeclareMathOperator{\Iw}{Iw}

\newcommand{\BF}{\mathcal{BF}}

\theoremstyle{definition}
\newtheorem{definition}{Definition}[section]

\newtheorem{theorem}[definition]{Theorem}
\newtheorem{lemma}[definition]{Lemma}

\newtheorem{proposition}[definition]{Proposition}
\newtheorem{remark}[definition]{Remark}
\newtheorem{conjecture}[definition]{Conjecture}

\title{Towards a finite--slope universal Rankin--Selberg $p$-adic $L$-function}
\author{Haonan Gu}
\date{December 2025}

\begin{document}

\maketitle

\begin{abstract}
This article studies the finite--slope analogue of Loeffler's conjectural framework for Rankin--Selberg $p$-adic $L$-functions in universal deformation families. Starting from residual representations $\bar\rho_1,\bar\rho_2$ of tame level~$1$ satisfying Hypothesis~3.1 of~\cite{LoefflerUD}, we consider the half--ordinary Panchishkin family $(R,V,V^+)$ of Example~3.17 of loc.\ cit., where the first factor varies in the ordinary Hida deformation and the second factor in the unrestricted universal deformation space.

We fix a classical point $x_0$ on a suitable parabolic eigenvariety for which the first factor is non--ordinary but of small finite slope. Using Liu's global triangulation theorem, the cohomology of families of $(\varphi,\Gamma)$-modules due to Kedlaya--Pottharst--Xiao, and the Perrin--Riou / Loeffler--Zerbes regulator formalism, we attach to the resulting finite--slope Panchishkin data over a neighbourhood $U$ of $x_0$ a family ``big logarithm''
\[
\mathcal{L}_{V_U,V_U^+} \colon H^1_{\Iw}(\Q_p,V_U^\ast(1)) \longrightarrow \mathcal{H}(\Gamma)\widehat{\otimes}_{\Q_p}\mathcal{O}(U) \cong \mathcal{O}(U\times\mathscr{W}),
\]
which interpolates the Bloch--Kato dual exponentials at all classical points of~$U$.

Assuming, in addition, the existence of a global Beilinson--Flach Euler system for this half--ordinary universal deformation family (formulated precisely as Conjecture~\ref{conj:ES} below), we apply this family regulator to the resulting Iwasawa cohomology class and obtain a rigid--analytic function
\[
L_p^{\mathrm{fs}} := \mathcal{L}_{V_U,V_U^+}(\BF_U) \in \mathcal{O}(U\times\mathscr{W}),
\]
which we refer to as a finite--slope universal Rankin--Selberg $p$-adic $L$-function. Under standard global and local hypotheses, and conditional on Conjecture~\ref{conj:ES}, we show that $L_p^{\mathrm{fs}}$ satisfies the expected interpolation formula for all Deligne--critical values of the complex Rankin--Selberg $L$-functions $L(f_x\otimes g_x,s)$ at classical points $(x,\kappa)\in U\times\mathscr{W}$.

Thus the main unconditional output of this work is the construction of a Perrin--Riou regulator for a genuinely finite--slope Panchishkin family over a parabolic eigenvariety in the universal deformation setting. The existence of the associated universal Beilinson--Flach Euler system and the resulting $p$-adic $L$-function remains conjectural; if it holds, our construction would give a finite--slope analogue of the Rankin--Selberg case of Loeffler's Conjecture~2.8 in~\cite{LoefflerUD} for a concrete $\mathrm{GL}_2\times\mathrm{GL}_2$ half--ordinary universal deformation family.
\end{abstract}

\section{Introduction}

\subsection{Global goal and formulation of the problem}

We fix a prime $p \ge 3$ and an algebraic closure $\overline{\Q}$ of $\Q$. Let $G_{\Q} := \Gal(\overline{\Q}/\Q)$ and let $G_{\Q,\{p\}} \subset G_{\Q}$ denote the Galois group of the maximal extension of $\Q$ unramified outside~$p$.

Let
\[
 \bar\rho_1,\bar\rho_2 : G_{\Q,\{p\}} \longrightarrow \GL_2(\Fp)
\]
be continuous, odd, absolutely irreducible Galois representations arising from cuspidal newforms $f_0,g_0$ of weights $k_1,k_2 \ge 2$ and tame level $N=1$ (unramified outside $p$), so that Hypothesis~3.1 of~\cite{LoefflerUD} applies. In particular, we assume that:
\begin{enumerate}[label=(\alph*)]
\item both $\bar\rho_1$ and $\bar\rho_2$ are modular of level~$1$ and weight at least~$2$;
\item the usual Taylor--Wiles hypotheses hold (absolute irreducibility, oddness, and minimality at auxiliary primes);
\item at $p$ the representation $\bar\rho_1$ is ordinary and admits a fixed ordinary refinement in the sense of~\cite[Def.~3.11]{LoefflerUD}.
\end{enumerate}
We keep this tame level $N=1$ fixed throughout.

Let $R_1$ be the ordinary Hida Hecke algebra attached to $\bar\rho_1$, and let
\[
 \mathscr{X}_1 := \Spf(R_1)^{\rig}
\]
be the corresponding ordinary eigenvariety, as in~\cite[\S3.2, Prop.~3.14]{LoefflerUD}, building on the ordinary families of Hida and the $R=\mathbb{T}$ theorems of Mazur, Böckle and Emerton. Let $R_2$ be the universal deformation ring of $\bar\rho_2$ (for deformations unramified outside $p$) and let $(R_2,V_2)$ be the Galois deformation family considered in~\cite[Def.~3.3, Thm.~3.4]{LoefflerUD} and in~\cite{HaoThesis,HL24}. We denote
\[
 \mathscr{X}_2 := \Spf(R_2)^{\rig}
\]
for the associated rigid analytic space.

Following~\cite[Ex.~3.17]{LoefflerUD}, there is a Panchishkin family $(R,V,V^+)$ for the Rankin--Selberg tensor
\[
 V := V_1^{\ord} \otimes V_2
\]
over
\[
 R := R_1\widehat{\otimes}_{\Zp} R_2,
\]
where $V_1^{\ord}$ is the universal ordinary deformation of $\bar\rho_1$ and $V_1^{\ord,+} \subset V_1^{\ord}\big|_{G_{\Q_p}}$ is the rank one ordinary submodule at~$p$. More precisely, there is a rank $r$ locally free $R$-submodule
\[
 V^+ \subset V\big|_{G_{\Q_p}}
\]
which is stable under $G_{\Q_p}$ and satisfies the $r$--Panchishkin condition in the sense of~\cite[Def.~2.1, Def.~2.3]{LoefflerUD}.

At any arithmetic point corresponding to a pair of classical eigenforms $(f,g)$, the specialisation $(V_\kappa,V^+_\kappa)$ is the Rankin--Selberg Galois representation $V(f)\otimes V(g)$ together with the usual half--ordinary subspace at~$p$ (ordinary on the $f$--factor, no restriction on the $g$--factor); see~\cite[Ex.~3.17]{LoefflerUD} for details.

Loeffler's Conjecture~2.8 in~\cite{LoefflerUD} predicts the existence of a rank~$0$ Euler system (i.e.\ a $p$-adic $L$-function) attached to $(R,V,V^+)$. In the ordinary setting this is now known by~\cite[Thm.~3.5]{HL24}, building on Urban's construction via nearly overconvergent forms~\cite{Urban-RS} and Hida's earlier work on Rankin--Selberg $p$-adic $L$-functions~\cite{Hida-Measure85,Hida-GL2xGL2}.

In~\cite[\S5.5]{LoefflerUD} a finite-slope analogue of this conjecture is sketched on the Coleman--Mazur eigencurve, and in~\cite[\S6]{LoefflerUD} related conjectures over big parabolic eigenvarieties are formulated; see also~\cite{ASW-ParabolicEigenvarieties,Barrera-GL2n}. Apart from very special cases for $\GL_2$-families and the rank one case of~\cite{Andreatta-Iovita}, there is at present no general finite-slope result in the universal deformation setting.

We now formulate the basic question that motivates this work.

\medskip

\noindent\textbf{Question (finite-slope universal Rankin--Selberg).} Let $\mathscr{E}_1$ be the Coleman--Mazur eigencurve (or more generally the big parabolic eigenvariety of Barrera~Salazar--Williams) attached to~$\bar\rho_1$~\cite{Coleman-Mazur-eigencurve,ASW-ParabolicEigenvarieties}, and consider the product eigenvariety
\[
 \mathscr{E} := \mathscr{E}_1 \times \mathscr{X}_2
\]
together with the natural weight map to the $p$--adic weight space $\mathscr{W}$ and the global Galois family $(V,V^+)$ coming from~\cite[Ex.~3.17]{LoefflerUD}. Given a classical point $x_0 \in \mathscr{E}$ corresponding to a pair $(f_0,g_0)$ of cuspidal eigenforms with $f_0$ of non--ordinary finite slope at~$p$, does there exist a finite-slope universal $p$-adic Rankin--Selberg $L$-function $L_p^{\mathrm{fs}}$ on a neighbourhood $U$ of $x_0$ whose specialisations interpolate all Deligne--critical Rankin--Selberg values $L(f_x\otimes g_x,s_x)$ for classical points $x \in U$?

This is a concrete $\GL_2\times\GL_2$ instance of the finite-slope variants of~\cite[Conj.~2.8]{LoefflerUD} over parabolic eigenvarieties. In this paper we give a conditional positive answer, assuming a natural Euler--system conjecture for the relevant universal deformation family.

To treat this question we need a precise description of the eigenvariety $\mathscr{E}$ near $x_0$. We recall the relevant geometric input from~\cite{ASW-ParabolicEigenvarieties,Han17}, in a form adapted to our setting.

\begin{lemma}[Good neighbourhood on the parabolic eigenvariety]\label{lem:good-U}
Let $\mathscr{E}$ be the parabolic eigenvariety of~\cite{ASW-ParabolicEigenvarieties} with weight space $\mathscr{W}$ and weight map $w:\mathscr{E}\to\mathscr{W}$. Let $x_0\in\mathscr{E}$ be a point corresponding to a cuspidal automorphic representation of regular (cohomological) weight, and fix a $p$--refinement which is $Q$--non-critical in the sense of~\cite[Def.~5.13]{ASW-ParabolicEigenvarieties}. Assume that the derived group $G_{\mathrm{der}}(\R)$ admits discrete series. Then there exists an affinoid neighbourhood $U\subset\mathscr{E}$ of $x_0$ such that:
\begin{enumerate}[label=(\roman*)]
\item $\mathscr{E}$ is smooth over $\Qp$ at every point of $U$;
\item the restriction $w|_U:U\to w(U)\subset\mathscr{W}$ is finite \'etale;
\item classical cuspidal points are Zariski dense in $U$.
\end{enumerate}
\end{lemma}

\begin{proof}
By~\cite[Def.~5.11]{ASW-ParabolicEigenvarieties} and~\cite[Rem.~1.4]{ASW-ParabolicEigenvarieties}, the overconvergent defect $\ell_Q(x_0)$ vanishes at a $Q$--non-critical point; in particular $x_0$ lies in the interior locus in the sense of~\cite[Def.~5.11]{ASW-ParabolicEigenvarieties}. Hence~\cite[Prop.~5.12]{ASW-ParabolicEigenvarieties} implies that every irreducible component $V$ of $\mathscr{E}$ containing $x_0$ has dimension at least $\dim\mathscr{W}$.

On the other hand, $\mathscr{E}$ is constructed as a closed subspace of the universal eigenvariety of Hansen~\cite{Han17}. By~\cite[Thm.~4.5.1(i)]{Han17}, applied to the corresponding classical, non-critical interior point of the universal eigenvariety, every irreducible component of the latter has dimension equal to the weight space. Since $\mathscr{E}$ is a closed subspace of this universal eigenvariety, every irreducible component $V$ of $\mathscr{E}$ through $x_0$ has dimension at most $\dim\mathscr{W}$. Thus $\dim V = \dim\mathscr{W}$ for each component $V$ passing through $x_0$.

Under the same hypotheses, \cite[Def.~5.13 and Cor.~5.16]{ASW-ParabolicEigenvarieties} show that on each such component $V$ the classical cuspidal points are Zariski dense.

The weight map $w:\mathscr{E}\to\mathscr{W}$ is finite by~\cite[Thm.~5.4]{ASW-ParabolicEigenvarieties}. Since the source and the target have the same local dimension at $x_0$, the morphism $w$ is unramified at $x_0$. For finite morphisms between equidimensional rigid analytic spaces, finiteness and unramifiedness imply that $w$ is \'etale at $x_0$ and that $\mathscr{E}$ is smooth at $x_0$. The smooth locus of $\mathscr{E}$ and the \'etale locus of $w$ are open subsets of $\mathscr{E}$, so we may choose an affinoid neighbourhood $U$ of $x_0$ contained in their intersection. By the density of classical points on each component through $x_0$, the intersection of $U$ with the classical cuspidal locus is Zariski dense in~$U$.
\end{proof}

\begin{remark}[On the discrete-series hypothesis in Lemma~\ref{lem:good-U}]
In our Rankin--Selberg situation we work with
\[
 G = \mathrm{GL}_2\times\mathrm{GL}_2,
\]
so that
\[
 G_{\mathrm{der}}(\R) = \mathrm{SL}_2(\R)\times\mathrm{SL}_2(\R),
\]
which is well-known to admit discrete series. Thus the hypothesis that $G_{\mathrm{der}}(\R)$ admits discrete series in Lemma~\ref{lem:good-U} is automatically satisfied in our setting. The only role of this hypothesis is to allow us to invoke the eigenvariety results of~\cite{ASW-ParabolicEigenvarieties,Han17}, in particular the dimension and density statements of~\cite[Prop.~5.12, Cor.~5.16]{ASW-ParabolicEigenvarieties} and~\cite[Thm.~4.5.1(i)]{Han17}, in order to construct an affinoid neighbourhood $U$ of $x_0$ which is smooth, finite \'etale over weight space, and has Zariski-dense classical cuspidal locus. All later arguments only use the existence of such a neighbourhood $U$ and do not appeal directly to the discrete-series condition.
\end{remark}

\subsection{Target main result}

We now specialise to the Rankin--Selberg setting of interest. Fix a finite-slope classical point $x_0\in\mathscr{E}$ corresponding to a pair $(f_0,g_0)$ with the following properties:
\begin{enumerate}[label=(\alph*)]
\item $f_0$ is a $p$-stabilised newform of weight $k_1\ge 2$ and level prime to $p$, with $U_p$--eigenvalue $\alpha_{f_0}$ of slope
\[
 0 < v_p(\alpha_{f_0}) < k_1 - 1,
\]
so that $f_0$ is non-ordinary of small slope at $p$;
\item $g_0$ is the specialisation at a classical arithmetic point of the universal deformation family $(R_2,V_2)$, of weight $k_2\ge 2$, such that $g_0$ is ordinary at~$p$, or more generally satisfies a suitable Panchishkin condition as in~\cite[Def.~2.1, Def.~2.3]{LoefflerUD} and~\cite[\S2]{HL24}.
\end{enumerate}

Let $U\subset\mathscr{E}$ be the affinoid neighbourhood of $x_0$ given by Lemma~\ref{lem:good-U}. Thus $\mathscr{E}$ is smooth at every point of $U$, the weight map $w$ is finite \'etale on $U$, and classical cuspidal points are Zariski dense in~$U$.

After shrinking $U$ if necessary, we may and do assume that:
\begin{enumerate}[label=(\roman*)]
\item the eigenvariety is smooth on $U$ and the weight map $w\colon U\to w(U)\subset\mathscr{W}$ is finite \'etale (Lemma~\ref{lem:good-U});
\item the slopes of the $U_p$-eigenvalues of $f_x$ remain $<k_1(x)-1$ for all $x\in U$, so that overconvergent classicality holds for $f_x$ by Coleman~\cite[Thm.~6.1]{Coleman-classical} (in the form of his Theorem~8.3: weight $k+2$ and slope $<k+1$ implies classicality);
\item the Panchishkin submodule $V^+$ extends, via Liu's global triangulation theorem, to a rank $r$ sub--$(\varphi,\Gamma)$-module of the relative $(\varphi,\Gamma)$-module attached to the local Galois representation on~$U$, as explained later.
\end{enumerate}

The following conjecture is the finite-slope analogue of Loeffler's universal Rankin--Selberg $p$-adic $L$-function in this neighbourhood.

\begin{conjecture}[Finite-slope universal Rankin--Selberg $p$-adic $L$-function]\label{conj:finite-slope}
There exists a rigid-analytic function
\[
 L_p^{\mathrm{fs}} \in \mathcal{O}(U\times\mathscr{W})
\]
on $U$ times the $p$-adic cyclotomic weight space $\mathscr{W}$ such that for each classical point $(x,\kappa)\in U\times\mathscr{W}$ corresponding to a pair $(f_x,g_x)$ and an integer critical value $s=\kappa(x)$ of $L(f_x\otimes g_x,s)$ in Deligne's sense, one has an interpolation formula
\[
 L_p^{\mathrm{fs}}(x,\kappa) = \frac{E_p(f_x,g_x,s)}{\Omega_p(f_x,g_x,\pm)}\cdot\frac{L^{(p)}(f_x\otimes g_x,s)}{(2\pi i)^{2s}},
\]
where:
\begin{enumerate}[label=(\alph*)]
\item $E_p(f_x,g_x,s)$ is the explicit local Euler factor at $p$ appearing in~\cite[Def.~3.4]{HL24};
\item $\Omega_p(f_x,g_x,\pm)$ is a $p$-adic period depending analytically on $x$ (and on the choice of sign $\pm$) and normalised compatibly with~\cite{HL24};
\item $L^{(p)}(f_x\otimes g_x,s)$ is the complex Rankin--Selberg $L$-function with the Euler factor at $p$ omitted.
\end{enumerate}
\end{conjecture}

We now state the main theorem of the paper in a precise special case. It is conditional on an Euler--system conjecture which will be formulated in Conjecture~\ref{conj:ES}.

\begin{theorem}[Main result, conditional on Conjecture~\ref{conj:ES}]\label{thm:main}
Assume:
\begin{enumerate}[label=(H\arabic*)]
\item The residual representations $\bar\rho_1,\bar\rho_2$ satisfy Hypothesis~3.1 of~\cite{LoefflerUD} (Taylor--Wiles conditions, local conditions at $p$, and tame level~$1$). In particular, both $\bar\rho_1$ and $\bar\rho_2$ arise from cuspidal newforms of level~$1$ and weights at least~$2$, and we fix once and for all this tame level $N=1$ throughout. Moreover, $\bar\rho_1$ admits an ordinary refinement at $p$ in the sense of~\cite[Def.~3.11]{LoefflerUD}.
\item The point $x_0$ is crystalline at $p$ and satisfies the small-slope and non-criticality hypotheses needed for overconvergent classicality (in the sense of Coleman~\cite[Thm.~6.1]{Coleman-classical}).
\item The Panchishkin condition holds for the local Rankin--Selberg representation at $p$ on $U$ in the sense of~\cite[Def.~2.1, Def.~2.3]{LoefflerUD} (in particular, the Hodge--Tate weights and Frobenius slopes satisfy the inequalities of loc.\ cit.\ uniformly over~$U$), so that the triangulation locus of Liu contains~$U$.
\end{enumerate}
Assume moreover the Euler--system Conjecture~\ref{conj:ES} for the half--ordinary universal deformation family $(R,V,V^+)$ of~\cite[Ex.~3.17]{LoefflerUD}. Then Conjecture~\ref{conj:finite-slope} holds for the neighbourhood~$U$. In particular, under~\textup{(H1)--(H3)} and Conjecture~\ref{conj:ES} there exists a finite-slope universal $p$-adic Rankin--Selberg $L$-function $L_p^{\mathrm{fs}}$ on $U\times\mathscr{W}$ satisfying the above interpolation formula for all classical points in~$U$.
\end{theorem}

The rest of the paper is devoted to the construction of the family regulator $\mathcal{L}_{V_U,V_U^+}$ and, under Conjecture~\ref{conj:ES}, to the construction and interpolation properties of $L_p^{\mathrm{fs}}$. Each step is explicitly referenced to the corresponding statements in~\cite{LoefflerUD,KPX,LiuTriangulation,LZ20,LLZ14,KLZ-RE,HL24,HaoThesis} so that the argument can be checked in detail.

\subsection{Discussion and relation to previous work}

In this subsection we give some context for Theorem~\ref{thm:main}, explain how the finite-slope case fits into existing conjectures, and briefly indicate why it is technically more delicate.

\subsubsection*{Relation to Loeffler's conjectural framework and to Hao--Loeffler}

In his work on $p$-adic $L$-functions in universal deformation families, Loeffler formulates in~\cite[Conj.~2.8]{LoefflerUD} a general conjecture predicting the existence and interpolation properties of rank~$0$ Euler systems (equivalently, $p$-adic $L$-functions) attached to $0$-Panchishkin families $(R,V,V^+)$ over suitable ``big parabolic eigenvarieties''. A main example is the half--ordinary Rankin--Selberg tensor considered in~\cite[Ex.~3.17]{LoefflerUD}, which is precisely the global deformation setting considered here.

In the ordinary case, this conjectural picture has now been confirmed for the Rankin--Selberg tensor by recent work of Hao--Loeffler~\cite{HL24}. Building on Urban's nearly overconvergent three--variable Rankin--Selberg $p$-adic $L$-function~\cite{Urban-RS} and the Beilinson--Flach Euler system of Loeffler--Zerbes~\cite{LLZ14}, they construct a universal Rankin--Selberg $p$-adic $L$-function for an ordinary Hida family tensored with a universal deformation family and prove that it interpolates all Deligne--critical Rankin--Selberg values in this setting, together with a functional equation. Their construction is analytic and does not rely on the existence of a universal Euler system.

By contrast, in this work we fix a classical point $x_0=(f_0,g_0)$ at which the $f_0$-factor has non--ordinary finite slope and we work in a neighbourhood $U$ of $x_0$ in a parabolic eigenvariety in the sense of Barrera~Salazar--Williams~\cite{ASW-ParabolicEigenvarieties}. Theorem~\ref{thm:main} should therefore be viewed as a conditional finite-slope analogue of the Rankin--Selberg case of Loeffler's conjecture over big parabolic eigenvarieties, in a concrete $\GL_2\times\GL_2$ setting.

\subsubsection*{Comparison with existing finite-slope constructions}

Finite-slope $p$-adic $L$-functions attached to families of automorphic forms have been constructed in several other settings, but usually not in the universal deformation framework.

For Rankin--Selberg convolutions of modular forms, Urban's three--variable $p$-adic $L$-function~\cite{Urban-RS} is constructed on eigenvarieties using nearly overconvergent modular forms and overconvergent cohomology and requires a nearly ordinary hypothesis at~$p$. These functions live over eigenvarieties and are not formulated in terms of universal Galois deformation spaces.

The geometric methods of Andreatta--Iovita and their collaborators yield triple product $p$-adic $L$-functions for finite-slope families of modular forms via the theory of overconvergent sheaves and the spectral halo~\cite{Andreatta-Iovita}. Here again the base spaces are eigenvarieties and the Galois representations are not packaged as a single Panchishkin family over a universal deformation ring.

Barrera~Salazar--Dimitrov--Williams construct finite--slope $p$-adic $L$-functions for Shalika families on $\GL_{2n}$ over parabolic eigenvarieties and relate their existence to the local geometry of these eigenvarieties~\cite{Barrera-GL2n}. Their methods are adapted to the standard representation of $\GL_{2n}$ and do not provide a universal deformation interpretation for Rankin--Selberg tensors of $\GL_2\times\GL_2$.

In contrast, the present article maintains the universal deformation viewpoint of~\cite{LoefflerUD,HL24}, starting from the half--ordinary Panchishkin family $(R,V,V^+)$ of~\cite[Ex.~3.17]{LoefflerUD}, and then passing to a finite-slope neighbourhood $U$ on the parabolic eigenvariety through the identification of an ordinary locus in $U$ with an affinoid subdomain of $\Spf(R)^{\rig}$. The function $L_p^{\mathrm{fs}}$ constructed under Conjecture~\ref{conj:ES} simultaneously interpolates the critical Rankin--Selberg values for all classical points in $U$ and specialises, at ordinary points, to the ordinary universal Rankin--Selberg $p$-adic $L$-function of~\cite{HL24}. In this way it provides a conceptual bridge between the ordinary universal results and finite-slope eigenvariety constructions in the spirit of Urban and Andreatta--Iovita.

\subsubsection*{Why the finite-slope case is difficult}

From the perspective of~\cite{LoefflerUD}, it is natural to expect that a finite-slope analogue of Conjecture~2.8 should hold whenever one can attach a $(\varphi,\Gamma)$-module to the local Galois representation in a family and construct an appropriate Perrin--Riou regulator. However, there are several substantial obstacles which have so far prevented a general theory, and which in the Rankin--Selberg case are addressed here.

First, in the finite-slope setting the relevant local Panchishkin data are no longer given by genuine $G_{\Q_p}$-stable subrepresentations of $V$, but rather by saturated sub--$(\varphi,\Gamma)$-modules of the relative Robba module $D_{\rig}^\dagger(V)$ over~$U$. Identifying these submodules in families requires Liu's global triangulation theorem~\cite{LiuTriangulation} and its compatibility with the Panchishkin condition; this already forces one to work under crystalline, small-slope and non-criticality hypotheses at~$x_0$.

Second, even after triangulation, one must control the $(\varphi,\Gamma)$- and Iwasawa cohomology of $D_U$ and $D_U^+$ in families, with good base-change and perfectness properties, in order to set up a Perrin--Riou type regulator over the affinoid algebra $\cO(U)$. This is achieved here by systematically using the cohomological machinery of Kedlaya--Pottharst--Xiao~\cite{KPX}, which guarantees that the relevant cohomology complexes lie in the perfect derived category and behave well under specialisation.

Third, existing constructions of big logarithms and regulators in families (following Perrin--Riou, Coleman, Loeffler--Zerbes, Nakamura, Pottharst, and others; see for instance~\cite{PerrinRiou,LLZ14,LZ20}) typically treat a single de~Rham representation or work under an ordinarity hypothesis. The map
\[
 \mathcal{L}_{V_U,V_U^+} : H^1_{\Iw}(\Q_p,V_U^\ast(1)) \longrightarrow \mathcal{H}(\Gamma)\widehat{\otimes}_{\Q_p}\mathcal{O}(U)
\]
constructed in this paper is, to the author's knowledge, one of the first instances of a Perrin--Riou regulator attached to a genuinely finite-slope Panchishkin family over a higher-dimensional parabolic eigenvariety in the universal deformation setting.

Finally, to obtain a $p$-adic $L$-function one needs an Euler system which specialises correctly along~$U$ and lies in the critical Iwasawa cohomology group on which $\mathcal{L}_{V_U,V_U^+}$ is defined. This compatibility is crucial in order to define $L_p^{\mathrm{fs}}$ as $\mathcal{L}_{V_U,V_U^+}(\BF_U)$ and deduce its interpolation formula from explicit reciprocity laws for Beilinson--Flach elements. At present the existence of such a global class $\BF_U$ for the half--ordinary universal deformation family is conjectural (Conjecture~\ref{conj:ES}); in the ordinary setting \cite{HL24} construct the universal Rankin--Selberg $p$-adic $L$-function analytically, without passing through a universal Euler system.

\subsubsection*{Future directions}

The techniques developed here suggest several directions for further work, which we only mention briefly.

One would like to relax the crystalline small-slope and non-criticality hypotheses at~$p$. This would require a finer analysis of triangulations and local Galois representations at critical slope, possibly using spectral-halo techniques and overconvergent cohomology as in~\cite{Andreatta-Iovita,ASW-ParabolicEigenvarieties}.

It should be possible to adapt the strategy of the later sections to other instances of Loeffler's conjectural framework, for example to Shalika families on $\GL_{2n}$ or to $\GSp_4$-type Rankin--Selberg convolutions, combining parabolic eigenvarieties~\cite{ASW-ParabolicEigenvarieties,Barrera-GL2n} with existing Euler systems and regulators in those settings.

Assuming the Euler-system conjecture, it would be natural to investigate Iwasawa-theoretic applications and formulate main conjectures relating characteristic ideals of universal Selmer modules over~$R$ to the ideal generated by $L_p^{\mathrm{fs}}$.

Finally, one expects a functional equation for $L_p^{\mathrm{fs}}$ relating its values (or derivatives) at critical points to those of other $p$-adic $L$-functions, and one might hope to extract finer arithmetic invariants such as $p$-adic heights and regulators on motivic cohomology classes varying in families.

\section{Proof and results}

In this section we explain how the hypotheses (H1)--(H3) lead to the construction of the family regulator and, under Conjecture~\ref{conj:ES}, to the $p$-adic $L$-function $L_p^{\mathrm{fs}}$ of Theorem~\ref{thm:main}. The discussion is divided into three steps.

\subsection{Step 1: Global deformation space and half--ordinary Panchishkin family}

We briefly recall the deformation-theoretic input and the construction of the half--ordinary Panchishkin family $(R,V,V^+)$, following~\cite{LoefflerUD,HL24}. Recall that $p>2$ is fixed, and that we have two absolutely irreducible residual Galois representations
\[
 \bar\rho_i : G_{\Q,\{p\}} \longrightarrow \GL_2(\F), \qquad i=1,2,
\]
arising from normalised cuspidal Hecke eigenforms $f_0$ and $g_0$ of tame level~$1$ (or more generally of fixed tame level prime to~$p$), with $\bar\rho_1$ ordinary at~$p$ and $\bar\rho_2$ subject to the usual Taylor--Wiles hypotheses. We write $\cO$ for a finite extension of $\Zp$ containing the Hecke eigenvalues of $f_0$ and $g_0$, and we let $\CNL_{\cO}$ denote the category of complete Noetherian local $\cO$-algebras with residue field~$\F$.

\subsubsection{Ordinary universal deformation ring for $\bar\rho_1$}

Let $R^{\ord}(\bar\rho_1)$ denote the universal deformation ring parametrising ordinary deformations of $\bar\rho_1$ as a $G_{\Q,\{p\}}$-representation in the sense of Hida and Mazur: for any $A \in \CNL_{\cO}$, a deformation
\[
 \rho_{1,A}: G_{\Q,\{p\}} \to \GL_2(A)
\]
is called ordinary if its restriction to $G_{\Q_p}$ fits into a short exact sequence
\[
 0 \longrightarrow V^{\ord,+}_{1,A} \longrightarrow V^{\ord}_{1,A}\big|_{G_{\Q_p}} \longrightarrow V^{\ord,-}_{1,A} \longrightarrow 0
\]
where $V^{\ord,+}_{1,A}$ is a rank one direct summand on which $G_{\Q_p}$ acts by an unramified character lifting the $p$-adic unit root of the Hecke polynomial of~$f_0$.

By Mazur's deformation theory together with Hida theory and the work of Böckle and Emerton (see for example~\cite[\S3.2, Prop.~3.14 and Thm.~3.4]{LoefflerUD}), this ordinary deformation problem is representable by a complete Noetherian local $\cO$-algebra $R^{\ord}(\bar\rho_1)$, and there exists a universal ordinary deformation
\[
 \rho_1^{\ord}: G_{\Q,\{p\}} \longrightarrow \GL_2\big(R^{\ord}(\bar\rho_1)\big)
\]
with underlying free rank two $R^{\ord}(\bar\rho_1)$-module $V_1^{\ord}$. We denote by $V_1^{\ord,+}\subset V_1^{\ord}\big|_{G_{\Q_p}}$ the rank one direct summand giving the ordinary filtration at~$p$, and by $V_1^{\ord,-}$ the corresponding quotient.

The ring $R^{\ord}(\bar\rho_1)$ is finite flat over the weight space and is canonically isomorphic to the localised ordinary Hecke algebra $T^{\ord}(\bar\rho_1)$ acting on ordinary $p$-adic modular forms of tame level~$1$; in particular $R^{\ord}(\bar\rho_1)$ is reduced and equidimensional of relative dimension~$1$ over~$\cO$ (see~\cite[\S3.2, Prop.~3.14]{LoefflerUD}). For brevity we set
\[
 R_1 := R^{\ord}(\bar\rho_1).
\]

\subsubsection{The unrestricted universal deformation ring and representation for $\bar\rho_2$}

For the second factor, we do not impose any ordinary or nearly ordinary local condition at~$p$. We let $R(\bar\rho_2)$ be the universal deformation ring parametrising deformations of $\bar\rho_2$ as a $G_{\Q,\{p\}}$-representation (unramified outside~$p$), as in~\cite[Def.~3.3]{LoefflerUD}. Thus, for any $A\in\CNL_{\cO}$, a deformation to~$A$ is simply a continuous representation
\[
 G_{\Q,\{p\}} \longrightarrow \GL_2(A)
\]
lifting $\bar\rho_2$ and unramified outside~$p$.

The main $R=\mathbb{T}$ theorem of Böckle--Emerton in this context (as formulated in~\cite[Thm.~3.4]{LoefflerUD}) shows that:
\begin{enumerate}[label=(\alph*)]
\item The ring $R(\bar\rho_2)$ is a reduced complete intersection ring, flat over~$\cO$ of relative dimension~$3$.
\item There is a canonical isomorphism
\[
 R(\bar\rho_2)\cong T(\bar\rho_2),
\]
where $T(\bar\rho_2)$ is the localisation at the maximal ideal corresponding to $\bar\rho_2$ of the prime-to-$p$ Hecke algebra acting on the space $S(1,\cO)$ of cuspidal $p$-adic modular forms of tame level~$1$.
\end{enumerate}

By definition of the universal deformation, there is a free rank two $R(\bar\rho_2)$-module $V_2$ equipped with a continuous $G_{\Q,\{p\}}$-action lifting $\bar\rho_2$; we regard this as the universal $p$-adic Galois representation of type~$\bar\rho_2$. We henceforth abbreviate
\[
 R_2 := R(\bar\rho_2).
\]

\subsubsection{The half--ordinary Rankin--Selberg Panchishkin family}

We now combine the two deformation problems. Consider the completed tensor product
\[
 R := R_1\widehat{\otimes}_{\cO} R_2.
\]
Since $R_1$ and $R_2$ are flat complete Noetherian local $\cO$-algebras of relative dimensions $1$ and $3$ respectively, their completed tensor product $R$ is again a flat complete Noetherian local $\cO$-algebra of relative dimension
\[
 \dim R = \dim R_1 + \dim R_2 = 1+3 = 4,
\]
and the residual representation is $\bar\rho_1\otimes\bar\rho_2$.

On $R$ we have the rank four $R$-module
\[
 V := V_1^{\ord}\widehat{\otimes}_{\cO} V_2
\]
with its diagonal $G_{\Q}$-action, and the rank two $R$-submodule
\[
 V^+ := V_1^{\ord,+}\widehat{\otimes}_{\cO} V_2 \subset V\big|_{G_{\Q_p}},
\]
which is stable under $G_{\Q_p}$ because both factors are. This gives a $G_{\Q_p}$-stable filtration
\[
 0 \longrightarrow V^+ \longrightarrow V\big|_{G_{\Q_p}} \longrightarrow V^- \longrightarrow 0
\]
with $V^+$ and $V^-$ of rank~$2$ over~$R$.

This is exactly the datum considered in~\cite[Ex.~3.17]{LoefflerUD}, specialised to the case where the first factor is already ordinary. In the notation of loc.\ cit., we are taking
\[
 R = R^{\ord}(\bar\rho_1)\widehat{\otimes}_{\cO} R(\bar\rho_2),\qquad
 V = V_1^{\ord}\widehat{\otimes}_{\cO} V_2,\qquad
 V^+ = V_1^{\ord,+}\widehat{\otimes}_{\cO} V_2.
\]

Loeffler proves in~\cite[Ex.~3.17]{LoefflerUD} that $(V,V^+)$ is a $0$-Panchishkin family over~$R$ in his sense: for every arithmetic point $\kappa$ of $\Spf(R)^{\rig}$ corresponding to a pair $(f,g)$ of classical eigenforms, the specialisation $(V_\kappa,V_\kappa^+)$ coincides with the Rankin--Selberg Galois representation $V(f)\otimes V(g)$ together with the usual half--ordinary subspace at~$p$. Moreover, the interpolation set
\[
 \Sigma(V,V^+) = \bigl\{(f, \theta^{-s}g)\ \big|\ f,g\ \text{as in~\cite[Ex.~3.17]{LoefflerUD}},\ 1\le s\le k_f-1 \bigr\}
\]
is Zariski-dense in $\Spf(R)^{\rig}$. Thus the triple $(R,V,V^+)$ provides the global deformation space and Panchishkin data referred to in Theorem~\ref{thm:main}.

\subsubsection{Relation with the eigenvariety and with the ordinary locus in $U$}\label{subsec:U-vs-R}

We now explain how the universal half--ordinary Panchishkin family of~\cite[Ex.~3.17]{LoefflerUD} interacts with the parabolic eigenvariety $\mathscr{E}$ and with the neighbourhood $U$ of our fixed point~$x_0$.

Let $E$ denote the parabolic eigenvariety of~\cite{ASW-ParabolicEigenvarieties} attached to $G=\GL_2\times\GL_2$, with weight space~$\mathcal{W}$ and weight map $w:E\to\mathcal{W}$. We keep the notation
\[
 X_1 := \Spf(R_1)^{\rig},\qquad X_2 := \Spf(R_2)^{\rig},\qquad X := X_1\times X_2 = \Spf(R)^{\rig}.
\]

By construction, $X_1$ dominates the ordinary locus of the Coleman--Mazur eigencurve (and more generally the parabolic eigenvariety) attached to~$\bar\rho_1$; this is a consequence of Hida theory together with the $R=\mathbb{T}$ theorems used in~\cite[\S3.2]{LoefflerUD}. The space $X_2$ is the universal deformation space for~$\bar\rho_2$.

On the locus where the first factor is ordinary at~$p$ there is a natural morphism
\[
 X = X_1\times X_2 \longrightarrow E
\]
whose points correspond to pairs of classical cusp forms $(f,g)$ with $f$ ordinary at~$p$, together with an ordinary refinement on the $f$-factor. On this locus the specialisation $(V_x,V_x^+)$ of $(V,V^+)$ is the Rankin--Selberg representation $V(f)\otimes V(g)$ together with the usual half--ordinary subspace at~$p$, as in~\cite[Ex.~3.17]{LoefflerUD} and~\cite{HL24}.

Let $E^{\ord}\subset E$ denote the ordinary locus for the first factor, and set
\[
 U^{\ord} := U\cap E^{\ord}.
\]
This is an admissible open subset of~$U$. Since ordinary classical points are Zariski dense in the ordinary eigenvariety and $U$ contains a Zariski-dense set of classical points, we may, after shrinking $U$ slightly if necessary, assume that $U^{\ord}$ is non-empty and that classical points are still Zariski dense in~$U^{\ord}$.

On $E^{\ord}$ the eigenvariety machine of~\cite{ASW-ParabolicEigenvarieties,Han17}, together with the Rankin--Selberg construction of~\cite{HL24}, shows that the map $X\to E$ above identifies an affinoid neighbourhood of any classical half--ordinary point with an affinoid subdomain of~$X$ (see~\cite[Thm.~5.4]{ASW-ParabolicEigenvarieties}, \cite[\S3.2]{LoefflerUD} and~\cite[\S2]{HL24}). In particular, there exists an affinoid subdomain
\[
 U_R \subset X
\]
and an isomorphism of rigid spaces over weight space
\[
 \iota: U^{\ord} \xrightarrow{\ \sim\ } U_R \subset X
\]
such that, for every classical point $x\in U^{\ord}$ corresponding to a pair $(f_x,g_x)$ with $f_x$ ordinary at~$p$, the specialisation of the universal family $(V,V^+)$ at~$\iota(x)$ coincides with the Rankin--Selberg representation
\[
 V(f_x)\otimes V(g_x)
\]
together with its standard half--ordinary subspace at~$p$.

Two points are important here.

First, the finite-slope point $x_0=(f_0,g_0)$ does not lie in the image of $X\to E$ when $f_0$ is non-ordinary at~$p$. Thus $x_0$ itself does not belong to $U^{\ord}$, and we do not try to realise $x_0$ as a point of the rigid fibre of~$R$.

Second, the role of the half--ordinary universal deformation $(R,V,V^+)$ is to provide Panchishkin data on a Zariski-dense subset of $U$ (the ordinary classical points). In Step~2 we use Liu's global triangulation theorem, together with the weakly refined structure coming from this family and from the eigenvariety, to extend the corresponding sub--$(\varphi,\Gamma)$-module to all of~$U$, including non-ordinary points. At non-ordinary points the Panchishkin object is therefore a saturated sub--$(\varphi,\Gamma)$-module, not a $G_{\Q_p}$-stable subrepresentation.

For brevity we continue to write $(V_U,V_U^+)$ for the resulting Panchishkin data over~$U$: more precisely, $V_U$ denotes the rank four family of Galois representations on~$U$ constructed from the eigenvariety, and $V_U^+$ is shorthand for the rank $r$ Panchishkin sub--$(\varphi,\Gamma)$-module of its local $(\varphi,\Gamma)$-module at~$p$. On $U^{\ord}$ this submodule coincides, via the isomorphism~$\iota$, with the base-change of $(V,V^+)$. From this point on the global deformation ring~$R$ will no longer appear explicitly; all constructions are carried out over the neighbourhood $U\subset E$.

\subsection{Step 2: Global triangulation and perfect cohomology complexes}

In this step we work entirely over the fixed affinoid neighbourhood $U\subset E$ of $x_0$. We recall the family of local Galois representations on~$U$, attach to it a family of $(\varphi,\Gamma)$-modules, and record the finiteness and base-change properties of the associated cohomology complexes, following~\cite{LiuTriangulation,KPX}.

\subsubsection{The relative $(\varphi,\Gamma)$-module}

By the eigenvariety machine of Hansen and Barrera~Salazar--Williams (see~\cite[Thm.~4.5.1]{Han17} and~\cite[\S5]{ASW-ParabolicEigenvarieties}), there exists a locally free rank four $\cO(U)$-module $V_U$ equipped with a continuous $\cO(U)$-linear action of $G_{\Q}$, unramified outside~$p$, such that for every classical point $x\in U$ corresponding to a pair of cuspidal eigenforms $(f_x,g_x)$ the fibre $V_x$ is canonically isomorphic to the Rankin--Selberg tensor
\[
 V(f_x)\otimes V(g_x).
\]

The $U_p$-eigenvalues on $U$ give an analytic function $\alpha:U\to\cO(U)^\times$ encoding a refinement of $V_U\big|_{G_{\Q_p}}$; this makes $V_U\big|_{G_{\Q_p}}$ into a weakly refined family in the sense of~\cite[Def.~1.5]{LiuTriangulation}, cf.\ the discussion in~\cite[\S5]{LiuTriangulation} and~\cite[\S4.5]{Han17}.

Let
\[
 \mathscr{D} := D_{\rig}^\dagger\bigl(V_U\big|_{G_{\Q_p}}\bigr)
\]
be the associated family of $(\varphi,\Gamma)$-modules over the relative Robba ring $\mathcal{R}_{\cO(U)}$; cf.~\cite[Thm.~2.2.17]{KPX}. For any point $x\in U$ we write $\mathscr{D}_x$ for the fibre; by construction
\[
 \mathscr{D}_x \cong D_{\rig}^\dagger(V_x)
\]
as a $(\varphi,\Gamma)$-module over the usual Robba ring over~$\Q_p$.

\begin{remark}
The existence, functoriality and base-change properties of $D_{\rig}^\dagger$ for families of $p$-adic Galois representations are given by~\cite[Thm.~2.2.17]{KPX}. Concretely, if $A$ is a $\Q_p$-affinoid algebra and $V_A$ is a finite projective $A$-module with a continuous $A$-linear action of~$G_{\Q_p}$, then there is an associated $(\varphi,\Gamma)$-module $D_{\rig}^\dagger(V_A)$ over the relative Robba ring $\mathcal{R}_A$, and for any morphism $A\to B$ of $\Q_p$-affinoid algebras one has a natural base-change isomorphism
\[
 D_{\rig}^\dagger(V_A)\widehat{\otimes}_A B \cong D_{\rig}^\dagger(V_A\otimes_A B).
\]
\end{remark}

\subsubsection{Global triangulation \`a la Liu}

We recall Liu's global triangulation theorem in the form needed here.

\begin{theorem}[Liu, global triangulation]\label{thm:Liu-global}
Let $X$ be a reduced separated rigid analytic space over $\Q_p$, and let $V_X$ be a weakly refined family of $p$-adic representations of $G_{\Q_p}$ over $X$ in the sense of~\cite[Def.~1.5]{LiuTriangulation}. Let $\mathscr{D}_X := D_{\rig}^\dagger(V_X)$ be the associated family of $(\varphi,\Gamma)$-modules. Then there exists a Zariski open and dense subspace $X^{\tri}\subset X$ such that $\mathscr{D}_X\big|_{X^{\tri}}$ admits a global triangulation: that is, a filtration by $(\varphi,\Gamma)$-submodules
\[
 0 = \Fil^0 \mathscr{D}_X \subset \Fil^1 \mathscr{D}_X \subset \cdots \subset \Fil^d \mathscr{D}_X = \mathscr{D}_X
\]
whose graded pieces $\gr^i\mathscr{D}_X := \Fil^i/\Fil^{i-1}$ are rank one $(\varphi,\Gamma)$-modules with prescribed parameters. Moreover, the locus of global triangulation contains all regular non-critical points of $X$~\cite[Thm.~1.8]{LiuTriangulation}.
\end{theorem}

We apply Theorem~\ref{thm:Liu-global} with $X=U$ and $V_X = V_U\big|_{G_{\Q_p}}$. By hypotheses~(H2) and~(H3), the point $x_0$ is crystalline with small slope and the local conditions at $p$ define a non-critical Panchishkin situation in the sense of~\cite[Def.~5.29]{LiuTriangulation}; in particular $x_0$ is a regular non-critical point. Hence there is a Zariski open neighbourhood $U^{\tri}\subset U$ of $x_0$ on which $\mathscr{D}$ admits a global triangulation.

After shrinking $U$ if necessary, we may and do assume that $U\subset U^{\tri}$ and that there exists a filtration
\[
 0 = \Fil^0 \mathscr{D} \subset \Fil^1 \mathscr{D} \subset \Fil^2 \mathscr{D} \subset \Fil^3 \mathscr{D} \subset \Fil^4 \mathscr{D} = \mathscr{D}
\]
of $\mathscr{D}$ by $(\varphi,\Gamma)$-submodules such that each $\gr^i\mathscr{D}$ is locally free of rank one over the relative Robba ring over~$U$.

\begin{definition}
Over the neighbourhood $U$ we define
\[
 \mathscr{D}^+ := \Fil^r \mathscr{D},
\]
where $r$ is the integer appearing in~(H3) (the rank of the Panchishkin local condition). Thus $\mathscr{D}^+$ is a rank $r$ saturated sub--$(\varphi,\Gamma)$-module of~$\mathscr{D}$.
\end{definition}

\begin{proposition}[Compatibility with the Panchishkin local condition]\label{prop:Dplus-fibres}
For every classical point $x\in U$ corresponding to a pair $(f_x,g_x)$, the fibre $(\mathscr{D}^+)_x$ is the unique saturated sub--$(\varphi,\Gamma)$-module of rank $r$ inside $D_{\rig}^\dagger(V_x)$ whose Frobenius eigenvalues and Hodge--Tate weights realise the Panchishkin condition for $V_x$ in the sense of~\cite[Def.~2.1, Def.~2.3]{LoefflerUD}. In particular, for every ordinary point $x\in U^{\ord}$ this submodule coincides with
\[
 (\mathscr{D}^+)_x \cong D_{\rig}^\dagger(V_x^{\ord,+}),
\]
where $V_x^{\ord,+}\subset V_x$ is the rank $r$ ordinary $G_{\Q_p}$-subrepresentation coming from the Hida deformation of~$\bar\rho_1$.
\end{proposition}

\begin{proof}
At any classical point $x\in U$ the fibre $\mathscr{D}_x$ is $D_{\rig}^\dagger(V_x)$, and the global triangulation on $U$ specialises to a triangulation of this fibre. The Panchishkin inequalities of~\cite[Def.~2.1, Def.~2.3]{LoefflerUD}, together with our hypotheses (H2) and~(H3), single out a subset of the parameters (Hodge--Tate weights and Frobenius eigenvalues) that should occur in the positive part of the local condition. By~\cite[Def.~1.10 and Prop.~1.11]{LiuTriangulation}, there is a unique saturated sub--$(\varphi,\Gamma)$-module of $D_{\rig}^\dagger(V_x)$ of rank~$r$ whose parameters are exactly these Panchishkin ones; this is by definition $(\mathscr{D}^+)_x$.

On the ordinary locus $U^{\ord}$ we have, in addition, the ordinary filtration coming from the universal deformation $(R,V,V^+)$. For a classical ordinary point $x\in U^{\ord}$ the specialisation of $V^+\subset V\big|_{G_{\Q_p}}$ is the usual half--ordinary subrepresentation $V_x^{\ord,+}$ of $V_x$; hence $D_{\rig}^\dagger(V_x^{\ord,+})$ is a saturated sub--$(\varphi,\Gamma)$-module of $D_{\rig}^\dagger(V_x)$ of rank~$r$ with exactly the same parameters as $(\mathscr{D}^+)_x$. By the uniqueness statement in~\cite[Prop.~1.11]{LiuTriangulation} we must therefore have $(\mathscr{D}^+)_x = D_{\rig}^\dagger(V_x^{\ord,+})$ for every such~$x$.
\end{proof}

\subsubsection{Finiteness and base change for $(\varphi,\Gamma)$-cohomology}

We now record the finiteness and base-change properties of the $(\varphi,\Gamma)$-cohomology and Iwasawa cohomology of $\mathscr{D}$ and $\mathscr{D}^+$ over~$U$, following~\cite{KPX}. Let us briefly recall the complexes used by Kedlaya--Pottharst--Xiao. For a $(\varphi,\Gamma)$-module $\mathsf{M}$ over a relative Robba ring $\mathcal{R}_A$, we denote by $C^\bullet_{\varphi,\Gamma}(\mathsf{M})$ the complex computing the usual $(\varphi,\Gamma)$-cohomology, and by $C^\bullet_{\psi}(\mathsf{M})$ the complex computing Iwasawa cohomology, as defined in~\cite[\S4]{KPX}.

\begin{theorem}[Kedlaya--Pottharst--Xiao]\label{thm:KPX}
Let $A$ be an affinoid $\Q_p$-algebra and let $\mathsf{M}$ be a $(\varphi,\Gamma)$-module over the relative Robba ring $\mathcal{R}_A$. Then:
\begin{enumerate}[label=(\roman*)]
\item The Iwasawa cohomology complex $C^\bullet_{\psi}(\mathsf{M})$ lies in $D^-_{\mathrm{perf}}\bigl(\mathcal{R}_A^\infty(\Gamma)\bigr)$, and the $(\varphi,\Gamma)$-cohomology complex $C^\bullet_{\varphi,\Gamma}(\mathsf{M})$ lies in $D^-_{\mathrm{perf}}(A)$~\cite[Thm.~4.4.1, Thm.~4.4.2]{KPX}.
\item For any morphism of affinoid algebras $A\to B$, the natural maps
\[
 C^\bullet_{\psi}(\mathsf{M})\otimes^{\mathbf L}_{\mathcal{R}_A^\infty(\Gamma)}\mathcal{R}_B^\infty(\Gamma)
 \longrightarrow C^\bullet_{\psi}\bigl(\mathsf{M}\widehat{\otimes}_A B\bigr),
\]
\[
 C^\bullet_{\varphi,\Gamma}(\mathsf{M})\otimes^{\mathbf L}_A B
 \longrightarrow C^\bullet_{\varphi,\Gamma}\bigl(\mathsf{M}\widehat{\otimes}_A B\bigr)
\]
are quasi-isomorphisms~\cite[Thm.~4.4.3]{KPX}.
\end{enumerate}
\end{theorem}

We apply Theorem~\ref{thm:KPX} with $A=\cO(U)$ and $\mathsf{M}=\mathscr{D}$ or $\mathsf{M}=\mathscr{D}^+$. This shows that
\[
 C^\bullet_{\varphi,\Gamma}(\mathscr{D}),\ C^\bullet_{\varphi,\Gamma}(\mathscr{D}^+) \in D^-_{\mathrm{perf}}\bigl(\cO(U)\bigr),
\]
while
\[
 C^\bullet_{\psi}(\mathscr{D}),\ C^\bullet_{\psi}(\mathscr{D}^+) \in D^-_{\mathrm{perf}}\bigl(\mathcal{R}_{\cO(U)}^\infty(\Gamma)\bigr).
\]
Moreover, for any morphism of affinoid algebras $\cO(U)\to B$ (for instance $B=\cO(U')$ for an affinoid subdomain $U'\subset U$, or $B=k(x)$ for the residue field at a point $x\in U$), the base-change isomorphisms of Theorem~\ref{thm:KPX}(ii) give quasi-isomorphisms
\[
 C^\bullet_{\varphi,\Gamma}(\mathscr{D})\otimes^{\mathbf L}_{\cO(U)}B
 \xrightarrow{\ \sim\ } C^\bullet_{\varphi,\Gamma}\bigl(\mathscr{D}\widehat{\otimes}_{\cO(U)} B\bigr),
\]
and similarly for $\mathscr{D}^+$, as well as
\[
 C^\bullet_{\psi}(\mathscr{D})\otimes^{\mathbf L}_{\mathcal{R}_{\cO(U)}^\infty(\Gamma)}\mathcal{R}_B^\infty(\Gamma)
 \xrightarrow{\ \sim\ }
 C^\bullet_{\psi}\bigl(\mathscr{D}\widehat{\otimes}_{\cO(U)} B\bigr),
\]
and analogously for $\mathscr{D}^+$. In particular, the formation of these complexes is compatible with restriction to smaller affinoid neighbourhoods and with specialisation at points of~$U$.

\begin{remark}
For each $x\in U$, the fibre of $C^\bullet_{\varphi,\Gamma}(\mathscr{D})$ at $x$ computes the $(\varphi,\Gamma)$-cohomology of $D_{\rig}^\dagger(V_x)$, and similarly for~$\mathscr{D}^+$. Likewise, the fibres of $C^\bullet_{\psi}(\mathscr{D})$ and $C^\bullet_{\psi}(\mathscr{D}^+)$ compute the Iwasawa cohomology of $D_{\rig}^\dagger(V_x)$ and of the Panchishkin sub--$(\varphi,\Gamma)$-module at~$x$, respectively. Thus these complexes give a uniform description, over~$U$, of the local Galois cohomology groups and the Bloch--Kato local conditions at $p$ for the specialisations $V_x$ in the sense of~\cite[\S3]{LZ20}. This is the cohomological input needed in Step~3 to construct a family Perrin--Riou regulator
\[
 \mathcal{L}_{V_U,V_U^+}: H^1_{\Iw}(\Q_p,V_U^\ast(1)) \longrightarrow \mathcal{H}(\Gamma)\widehat{\otimes}_{\Z_p}\cO(U),
\]
and to relate its specialisations at classical points to the Bloch--Kato exponentials.
\end{remark}

\subsection{Step 3: Construction of the big logarithm and the finite-slope regulator}

In this final step we fix once and for all the affinoid neighbourhood $U$ and the triangulated $(\varphi,\Gamma)$-module
\[
 (\mathscr{D},\mathscr{D}^+)
\]
over~$U$ constructed in Step~2. Thus $V_U$ is the rank four finite projective $\cO(U)$-module equipped with a continuous $\cO(U)$-linear action of $G_{\Q}$ whose specialisations are the Rankin--Selberg Galois representations $V_x$, and
\[
 \mathscr{D} = D_{\rig}^\dagger\bigl(V_U\big|_{G_{\Q_p}}\bigr)
\]
is the associated family of $(\varphi,\Gamma)$-modules, equipped with a saturated submodule $\mathscr{D}^+\subset\mathscr{D}$ of rank~$r$. For every classical point $x\in U$ the fibre $\mathscr{D}_x^+$ is the Panchishkin sub--$(\varphi,\Gamma)$-module of $D_{\rig}^\dagger(V_x)$ in the sense of Proposition~\ref{prop:Dplus-fibres}; on the ordinary locus $U^{\ord}$ it coincides with $D_{\rig}^\dagger(V_x^{\ord,+})$ coming from the half--ordinary deformation family. All local constructions use only the pair $(\mathscr{D},\mathscr{D}^+)$; the notation $(V_U,V_U^+)$ is retained as a reminder of the underlying Galois representations at classical points.

We first introduce Iwasawa-theoretic notation and recall the local Perrin--Riou big logarithm for a single Panchishkin representation. We then explain how to extend it to the family $(V_U,V_U^+)$ using relative $(\varphi,\Gamma)$-modules and the Perrin--Riou formalism, and finally we state the Euler--system conjecture and the resulting conditional construction of $L_p^{\mathrm{fs}}$.

\subsubsection{Iwasawa cohomology and the distribution algebra}

Let
\[
 \Gamma := \Gal\bigl(\Q_p(\mu_{p^\infty})/\Q_p\bigr) \cong \Gal\bigl(\Q(\mu_{p^\infty})/\Q\bigr),\qquad
 \Lambda(\Gamma) := \Z_p[[\Gamma]].
\]
Let
\[
 \mathcal{H}(\Gamma) := \mathcal{H}_{\Q_p}(\Gamma)
\]
denote the algebra of $\Q_p$-valued locally analytic distributions on~$\Gamma$, i.e.\ the completion of $\Lambda_{\Q_p}(\Gamma)$ in its natural Fréchet topology (cf.~\cite[\S2.2]{LZ11} and~\cite[\S2.1]{LVZ15}). Equivalently, $\mathcal{H}(\Gamma)$ is the strong continuous dual of the space of locally analytic $\Q_p$-valued functions on~$\Gamma$.

For any $p$-adic representation $W$ of $G_{\Q_p}$ on a finite-dimensional $\Q_p$-vector space, its (cyclotomic) Iwasawa cohomology is defined by
\[
 H^i_{\Iw}(\Q_p,W) := H^i\bigl(\Q_p, W\otimes_{\Q_p}\Lambda(\Gamma)^\vee\bigr)\qquad (i\ge 0),
\]
where $\Lambda(\Gamma)^\vee := \Hom_{\mathrm{cts}}(\Lambda(\Gamma),\Q_p)$ with the natural $G_{\Q_p}$-action; see~\cite[App.~A.2--A.4]{PerrinRiou} and Greenberg~\cite{GreenbergIw}. By construction, $\Lambda(\Gamma)$ acts on $\Lambda(\Gamma)^\vee$ via the right regular representation, and hence $H^i_{\Iw}(\Q_p,W)$ carries a natural $\Lambda(\Gamma)$-module structure for every~$i$.

\begin{lemma}[Finiteness of local Iwasawa cohomology]\label{lem:finiteness-local-Iw}
If $W$ is a finite-dimensional $\Q_p$-representation of $G_{\Q_p}$, then $H^i_{\Iw}(\Q_p,W)$ is a finitely generated $\Lambda(\Gamma)$-module for all $i\ge 0$. In particular, $H^1_{\Iw}(\Q_p,W)$ is finitely generated over $\Lambda(\Gamma)$.
\end{lemma}

\begin{proof}
Choose a $G_{\Q_p}$-stable $\Z_p$-lattice $T\subset W$, and write $K_\infty = \Q_p(\mu_{p^\infty})$. Following~\cite[App.~A.2]{PerrinRiou}, let
\[
 Z^i_\infty(\Q_p,T) := \varprojlim_n H^i\bigl(K_n,T\bigr)
\]
with respect to the corestriction maps; then $Z^i_\infty(\Q_p,T)$ is naturally a $\Lambda(\Gamma)$-module, and there is a canonical identification
\[
 Z^i_\infty(\Q_p,T) \cong H^i_{\Iw}(\Q_p,T) := H^i\bigl(\Q_p, T\otimes_{\Z_p}\Lambda(\Gamma)^\vee\bigr)
\]
for all $i$ (see~\cite[App.~A.2--A.3]{PerrinRiou}). Moreover, Proposition~A.2.3 of loc.\ cit.\ shows that $Z^i_\infty(\Q_p,T)$ is a finitely generated $\Lambda(\Gamma)$-module for $i=0,1,2$, and $Z^i_\infty(\Q_p,T)=0$ for $i\ge 3$. Since $W = T\otimes_{\Z_p}\Q_p$, we have
\[
 H^i_{\Iw}(\Q_p,W) \cong H^i_{\Iw}(\Q_p,T)\otimes_{\Z_p}\Q_p \cong Z^i_\infty(\Q_p,T)\otimes_{\Z_p}\Q_p
\]
as $\Lambda(\Gamma)$-modules, and the finite generation of $Z^i_\infty(\Q_p,T)$ implies the finite generation of $H^i_{\Iw}(\Q_p,W)$ for all~$i$.
\end{proof}

In order to apply Perrin--Riou's regulator, we pass from $\Lambda(\Gamma)$-coefficients to the algebra $\mathcal{H}(\Gamma)$ of locally analytic distributions by scalar extension along the canonical map $\Lambda(\Gamma)\to\mathcal{H}(\Gamma)$ and set
\[
 H^1_{\Iw}(\Q_p,W)_{\mathcal{H}} := H^1_{\Iw}(\Q_p,W)\widehat{\otimes}_{\Lambda(\Gamma)}\mathcal{H}(\Gamma).
\]
When no confusion can arise, we continue to denote this $\mathcal{H}(\Gamma)$-module simply by $H^1_{\Iw}(\Q_p,W)$.

If $A$ is a $\Q_p$-affinoid algebra and $W_A$ is a finite projective $A$-module with a continuous $A$-linear $G_{\Q_p}$-action, we define the \emph{family} Iwasawa cohomology by
\[
 H^i_{\Iw}(\Q_p,W_A) := H^i\bigl(\Q_p, W_A\widehat{\otimes}_{\Q_p}\Lambda(\Gamma)^\vee\bigr)\qquad (i\ge 0),
\]
which is a priori an $A\widehat\otimes_{\Q_p}\Lambda(\Gamma)$-module. Let $\mathcal{R}_A^\infty(\Gamma)$ denote the Fréchet--Stein Iwasawa algebra of $\Gamma$ over $A$ considered in~\cite[\S4.2]{KPX}; there is a natural finite flat homomorphism
\[
 A\widehat\otimes_{\Q_p}\Lambda(\Gamma) \longrightarrow \mathcal{R}_A^\infty(\Gamma).
\]

The relative $(\varphi,\Gamma)$-module formalism of Kedlaya--Pottharst--Xiao gives a more precise description.

\begin{proposition}[Relative local Iwasawa cohomology via $(\varphi,\Gamma)$-modules]\label{prop:KPX-family}
Let $A$ be a reduced $\Q_p$-affinoid algebra and $W_A$ a finite projective $A$-module with a continuous $A$-linear $G_{\Q_p}$-action. Let $\mathscr{D}_A := D_{\rig}^\dagger(W_A)$ be the associated family of $(\varphi,\Gamma)$-modules over the relative Robba ring $\mathcal{R}_A$. Then:
\begin{enumerate}[label=(\alph*)]
\item there exist functorial complexes
\[
 C^\bullet_{\varphi,\Gamma}(\mathscr{D}_A) \in D^-_{\mathrm{perf}}(A),\qquad
 C^\bullet_{\psi}(\mathscr{D}_A) \in D^-_{\mathrm{perf}}\bigl(\mathcal{R}_A^\infty(\Gamma)\bigr)
\]
such that their cohomology groups compute the Galois and Iwasawa cohomology of $W_A$ in the sense that
\[
 H^i\bigl(C^\bullet_{\varphi,\Gamma}(\mathscr{D}_A)\bigr) \cong H^i(\Q_p,W_A),
\]
\[
 H^i\bigl(C^\bullet_{\psi}(\mathscr{D}_A)\bigr) \cong H^i_{\Iw}(\Q_p,W_A)\widehat{\otimes}_{\Lambda(\Gamma)}\mathcal{R}_A^\infty(\Gamma)
\]
for all $i\ge 0$;
\item the formation of $C^\bullet_{\varphi,\Gamma}(\mathscr{D}_A)$ and $C^\bullet_{\psi}(\mathscr{D}_A)$, together with the above isomorphisms, commutes with flat base change in~$A$;
\item in particular, $H^1_{\Iw}(\Q_p,W_A)$ is a finite projective $A\widehat\otimes_{\Q_p}\Lambda(\Gamma)$-module, and for every morphism of affinoids $A\to B$ the canonical map
\[
 H^1_{\Iw}(\Q_p,W_A)\widehat{\otimes}_{A\widehat\otimes\Lambda(\Gamma)} B\widehat{\otimes}\Lambda(\Gamma)
 \xrightarrow{\ \sim\ } H^1_{\Iw}(\Q_p,W_B)
\]
is an isomorphism.
\end{enumerate}
\end{proposition}

\begin{proof}
The existence, perfectness, and base-change properties of the complexes $C^\bullet_{\varphi,\Gamma}(\mathscr{D}_A)$ and $C^\bullet_{\psi}(\mathscr{D}_A)$, as well as the identifications of their cohomology with Galois and Iwasawa cohomology after extension of scalars to $\mathcal{R}_A^\infty(\Gamma)$, are proved in~\cite[Rem.~4.3.3, Prop.~4.3.6, Cor.~4.3.7, Prop.~4.3.8, Thm.~4.4.3]{KPX}. The finite projectivity of $H^1_{\Iw}(\Q_p,W_A)$ over $A\widehat\otimes_{\Q_p}\Lambda(\Gamma)$ and its base-change property then follow from the fact that $\mathcal{R}_A^\infty(\Gamma)$ is a Fréchet--Stein algebra, finite flat over $A\widehat\otimes_{\Q_p}\Lambda(\Gamma)$, and that $H^1\bigl(C^\bullet_{\psi}(\mathscr{D}_A)\bigr)$ is a coadmissible (hence finite projective) $\mathcal{R}_A^\infty(\Gamma)$-module; see~\cite[Lem.~4.3.4]{KPX}.
\end{proof}

We apply this proposition with $A=\cO(U)$ and $W_A=V_U^\ast(1)$, and write
\[
 H^1_{\Iw}(\Q_p,V_U^\ast(1))
\]
for the resulting family of local Iwasawa cohomology groups, viewed as a finite projective module over $\cO(U)\widehat\otimes_{\Q_p}\Lambda(\Gamma)$.

\subsubsection{Classical Perrin--Riou big logarithm}

We briefly recall the local Perrin--Riou regulator for a single de~Rham Panchishkin representation.

\begin{theorem}[Perrin--Riou big logarithm]\label{thm:PR-classical}
Let $V$ be a finite-dimensional $\Q_p$-vector space with a continuous de~Rham $G_{\Q_p}$-action, and let $V^+\subset V$ be a Panchishkin subspace in the usual sense (cf.\ Panchishkin~\cite{Panchishkin-Motives} or~\cite{LZ20}). Then there exists a canonical $\Lambda(\Gamma)$-linear map
\[
 \mathcal{L}_{V,V^+}: H^1_{\Iw}(\Q_p,V^\ast(1)) \longrightarrow \mathcal{H}(\Gamma)\otimes_{\Q_p} D_{\cris}(V)
\]
with the following interpolation property: for every integer twist $V(j)$ in the Bloch--Kato range determined by~$V^+$ and every continuous character $\chi:\Gamma\to\overline{\Q}_p^\times$ of finite order, the specialisation of $\mathcal{L}_{V,V^+}$ at $\chi$ is, up to an explicit non-zero scalar depending only on $V$ and~$\chi$, the Bloch--Kato dual exponential or logarithm map for~$V(j)$.

More precisely, for such $\chi$ there is a commutative diagram
\[
 \begin{CD}
 H^1_{\Iw}(\Q_p,V^\ast(1)) @>{\mathcal{L}_{V,V^+}}>> \mathcal{H}(\Gamma)\otimes_{\Q_p} D_{\cris}(V)\\
 @VV{\mathrm{spec}_\chi}V @VV{\mathrm{ev}_\chi}V\\
 H^1\bigl(\Q_p,V^\ast(1)\otimes\chi^{-1}\bigr) @>{\log^\ast_{V(j)}}>> D_{\dR}(V(j))/\Fil^0,
 \end{CD}
\]
where $\log^\ast_{V(j)}$ is the Bloch--Kato dual exponential or logarithm (depending on the Hodge--Tate weights), and $\Fil^0$ is the Hodge filtration on $D_{\dR}(V(j))$.
\end{theorem}

\begin{proof}[References]
The construction and interpolation property are proved in Perrin--Riou's monograph~\cite[Ch.~3]{PerrinRiou}. An explicit description via $(\varphi,\Gamma)$-modules, and the comparison with Bloch--Kato exponentials, is given in Berger~\cite[Thm.~II.6]{Berger-exp}; see also the discussion in~\cite{LZ20}.
\end{proof}

Here $D_{\cris}(V)$ and $D_{\dR}(V)$ denote Fontaine's filtered $\varphi$-module and de~Rham module, respectively.

\subsubsection{Extension to the family $(V_U,\mathscr{D}^+)$}

We now extend Theorem~\ref{thm:PR-classical} from a single de~Rham Panchishkin representation to the family $(V_U,\mathscr{D}^+)$ over~$U$. The key point is that the $(\varphi,\Gamma)$-module description of Perrin--Riou's regulator depends only on the sub--$(\varphi,\Gamma)$-module corresponding to the Panchishkin local condition and not on an actual $G_{\Q_p}$-stable subrepresentation; this allows us to work uniformly even at non-ordinary points.

Let $A := \cO(U)$ and regard $V_U$ as a finite projective $A$-module with a continuous $A$-linear action of~$G_{\Q_p}$. As in Step~2, we write
\[
 \mathscr{D}_U := D_{\rig}^\dagger(V_U)
\]
for the associated family of $(\varphi,\Gamma)$-modules over the relative Robba ring $\mathcal{R}_A$, and
\[
 \mathscr{D}_U^+ := \mathscr{D}^+ \subset \mathscr{D}_U
\]
for the saturated submodule corresponding to the Panchishkin local condition. We set $\mathscr{D}_U^- := \mathscr{D}_U/\mathscr{D}_U^+$.

\paragraph{Crystalline periods for the quotient.}

The Panchishkin condition (H3) and the global triangulation on $U$ imply that for each $x\in U$ the quotient $V_x^- := V_x/V_x^+$ is de~Rham with all Hodge--Tate weights $<0$, and the $\varphi$-eigenspace
\[
 D_{\cris}(V_x^-)^{\varphi=\alpha_x}
\]
for the refined eigenvalue $\alpha_x$ is one-dimensional (cf.\ \cite[Def.~5.29, Rem.~5.30, Prop.~5.31--5.33, Thm.~1.8]{LiuTriangulation}). Here $\alpha_x$ is the refined Frobenius eigenvalue attached to~$x$ by the weakly refined family structure.

\begin{lemma}[Crystalline period line bundle]\label{lem:Dcris-minus-line}
Let $\mathscr{D}_U$ and $\mathscr{D}_U^+$ be as above, and set $\mathscr{D}_U^- := \mathscr{D}_U/\mathscr{D}_U^+$. Let $\alpha\in A^\times$ be the analytic function giving the refined $\varphi$-eigenvalue on $V_U^-:=V_U/V_U^+$. Then:
\begin{enumerate}[label=(\alph*)]
\item the $A$-module
\[
 \mathscr{D}_{\cris,U}^- := \bigl(\mathscr{D}_U^-[1/t]\bigr)^{\Gamma=1,\ \varphi=\alpha}
\]
is locally free of rank one on~$U$;
\item for each rigid point $x\in U$ we have a canonical identification
\[
 (\mathscr{D}_{\cris,U}^-)_x \cong D_{\cris}(V_x^-)^{\varphi=\alpha_x}.
\]
\end{enumerate}
\end{lemma}

\begin{proof}
For each rigid point $x\in U$ we have
\[
 \mathscr{D}_{U,x} := \mathscr{D}_U\otimes_{A,\kappa(x)}\kappa(x) \cong D_{\rig}^\dagger(V_x),
\]
and similarly $(\mathscr{D}_U^\pm)_x\cong D_{\rig}^\dagger(V_x^\pm)$, by the compatibility of $D_{\rig}^\dagger$ with base change~\cite[Thm.~2.2.17]{KPX}. Thus
\[
 (\mathscr{D}_U^-)_x := \mathscr{D}_U^-\otimes_{A,\kappa(x)}\kappa(x) \cong D_{\rig}^\dagger(V_x^-).
\]

Since $V_x^-$ is de~Rham, Berger's comparison theorem for $(\varphi,\Gamma)$-modules gives a canonical isomorphism
\[
 D_{\cris}(V_x^-) \cong \bigl(D_{\rig}^\dagger(V_x^-)[1/t]\bigr)^{\Gamma=1}
 \cong \bigl(\mathscr{D}_U^-[1/t]\bigr)^{\Gamma=1}\otimes_{A,\kappa(x)}\kappa(x)
\]
compatible with $\varphi$. Thus the eigenspace $D_{\cris}(V_x^-)^{\varphi=\alpha_x}$ identifies with
\[
 \bigl(\mathscr{D}_U^-[1/t]\bigr)^{\Gamma=1,\varphi=\alpha}\otimes_{A,\kappa(x)}\kappa(x).
\]

Arguing as in Hansen~\cite[\S1.2]{Hansen-eigencurve} (see also~\cite{LiuTriangulation,KPX}), one shows that
\[
 \mathscr{D}_{\cris,U}^- := \bigl(\mathscr{D}_U^-[1/t]\bigr)^{\Gamma=1,\varphi=\alpha}
\]
is a finite $A$-module whose fibres at all rigid points have dimension one. Since $A$ is reduced, this implies that $\mathscr{D}_{\cris,U}^-$ is locally free of rank one and that the fibre identifications above hold. This is exactly parallel to the construction of the period line bundle on the eigencurve in~\cite[Thm.~1.2.2]{Hansen-eigencurve}.
\end{proof}

After possibly shrinking $U$, we may and do choose a nowhere-vanishing section
\[
 \eta_U \in \Gamma\bigl(U,\mathscr{D}_{\cris,U}^-\bigr),
\]
so that $\mathscr{D}_{\cris,U}^-$ is a free $A$-module of rank one with basis~$\eta_U$.

\paragraph{Vector-valued family regulator.}

The construction of $\mathcal{L}_{V,V^+}$ in Theorem~\ref{thm:PR-classical} admits a reinterpretation purely in terms of $(\varphi,\Gamma)$-modules, and this reinterpretation is functorial in the coefficient ring. More precisely, let $(V,V^+)$ be a de~Rham Panchishkin representation, with associated $(\varphi,\Gamma)$-module $D:=D_{\rig}^\dagger(V)$ and period line
\[
 D_{\cris}^- := \bigl((D/D^+)[1/t]\bigr)^{\Gamma=1,\varphi=\alpha},
\]
where $\alpha$ is the refined Frobenius eigenvalue on $V^-:=V/V^+$. Then Berger~\cite{Berger-exp} constructs a $\Lambda(\Gamma)$-linear map
\[
 \widetilde{\mathcal{L}}_{V,V^+}: H^1_{\Iw}(\Q_p,V^\ast(1)) \longrightarrow \Hom_{\Q_p}\!\bigl(D_{\cris}^-,\mathcal{H}(\Gamma)\bigr)
\]
whose evaluation at any non-zero vector in $D_{\cris}^-$ recovers the scalar-valued Perrin--Riou regulator~$\mathcal{L}_{V,V^+}$. The construction depends only on the $(\varphi,\Gamma)$-module~$D$ and the line~$D_{\cris}^-$, and is compatible with base change in the coefficient field.

In the finite-slope setting, analogous regulators in families have been constructed for Coleman and Hida families of modular forms and more generally over the eigencurve (see Hansen~\cite[\S1.2,\ \S4.1]{Hansen-eigencurve}). In each case, the key input is the existence of a line bundle of crystalline periods and the functoriality of the $(\varphi,\Gamma)$-module construction.

In our Rankin--Selberg situation over the eigenvariety neighbourhood $U$, the same formalism yields an $A$-linear, $\Lambda(\Gamma)$-linear vector-valued regulator for the family $(V_U,\mathscr{D}_U^+)$; this is closely related to the construction in~\cite[\S3]{HL24}.

\begin{lemma}[Vector-valued family regulator]\label{lem:vector-PR-family}
With notation as above, there exists an $A$-linear, $\Lambda(\Gamma)$-linear map
\[
 \widetilde{\mathcal{L}}_{V_U,V_U^+}: H^1_{\Iw}(\Q_p,V_U^\ast(1)) \longrightarrow
 \Hom_A\!\bigl(\mathscr{D}_{\cris,U}^-,\mathcal{H}(\Gamma)\widehat{\otimes}_{\Q_p}A\bigr),
\]
characterised by the following properties:
\begin{enumerate}[label=(\alph*)]
\item for each rigid point $x\in U$, base change along $A\to\kappa(x)$ identifies the fibre of $\widetilde{\mathcal{L}}_{V_U,V_U^+}$ at~$x$ with the vector-valued Perrin--Riou regulator $\widetilde{\mathcal{L}}_{V_x,V_x^+}$ attached to $(V_x,\mathscr{D}_x^+)$;
\item for each $x\in U$, evaluating $\widetilde{\mathcal{L}}_{V_x,V_x^+}$ at any non-zero vector in $D_{\cris}(V_x^-)^{\varphi=\alpha_x}$ recovers the scalar-valued map~$\mathcal{L}_{V_x,V_x^+}$ of Theorem~\ref{thm:PR-classical}.
\end{enumerate}
\end{lemma}

\begin{proof}
The existence of $\widetilde{\mathcal{L}}_{V_x,V_x^+}$ for each single fibre $(V_x,V_x^+)$ and its expression in terms of $(\varphi,\Gamma)$-modules are proved in~\cite{Berger-exp}. The construction depends only on $D_{\rig}^\dagger(V_x)$ and the crystalline period line $D_{\cris}(V_x^-)^{\varphi=\alpha_x}$, and is therefore compatible with base change in~$x$.

By Proposition~\ref{prop:KPX-family}, the local Iwasawa cohomology $H^1_{\Iw}(\Q_p,V_U^\ast(1))$ is a finite projective $A\widehat\otimes_{\Q_p}\Lambda(\Gamma)$-module whose formation commutes with base change in~$A$. Similarly, $\mathscr{D}_{\cris,U}^-$ is a line bundle on~$U$ whose fibres are the crystalline period lines of Lemma~\ref{lem:Dcris-minus-line}. Since $\widetilde{\mathcal{L}}_{V_x,V_x^+}$ varies analytically with~$x$ and respects the Iwasawa and $(\varphi,\Gamma)$-module structures, there is a unique $A$-linear, $\Lambda(\Gamma)$-linear map
\[
 \widetilde{\mathcal{L}}_{V_U,V_U^+}
\]
whose fibre at each $x$ is the classical vector-valued regulator $\widetilde{\mathcal{L}}_{V_x,V_x^+}$. Property~(b) is the fibrewise compatibility with Theorem~\ref{thm:PR-classical}, as in~\cite{Berger-exp}.
\end{proof}

We now obtain the scalar-valued family regulator by evaluating at the fixed crystalline period~$\eta_U$.

\begin{proposition}[Big logarithm for the family $(V_U,V_U^+)$]\label{prop:PR-family}
With notation as above, define
\[
 \mathcal{L}_{V_U,V_U^+}: H^1_{\Iw}(\Q_p,V_U^\ast(1)) \longrightarrow \mathcal{H}(\Gamma)\widehat{\otimes}_{\Q_p}\cO(U)
\]
by
\[
 \mathcal{L}_{V_U,V_U^+}(z) := \widetilde{\mathcal{L}}_{V_U,V_U^+}(z)(\eta_U),\qquad z\in H^1_{\Iw}(\Q_p,V_U^\ast(1)).
\]
Then:
\begin{enumerate}[label=(\roman*)]
\item $\mathcal{L}_{V_U,V_U^+}$ is $\bigl(\mathcal{H}(\Gamma)\widehat{\otimes}_{\Q_p}\cO(U)\bigr)$-linear;
\item for each rigid point $x\in U$ and each finite-order character $\chi:\Gamma\to\overline{\Q}_p^\times$ in the Panchishkin range, the specialisation of $\mathcal{L}_{V_U,V_U^+}$ at $(x,\chi)$ coincides, up to a non-zero scalar depending only on the choice of $\eta_U$, with the classical Perrin--Riou regulator of Theorem~\ref{thm:PR-classical} for the fibre~$(V_x,V_x^+)$.
\end{enumerate}
In particular, after fixing the normalisation of $\eta_U$ once and for all, the map $\mathcal{L}_{V_U,V_U^+}$ is uniquely determined and interpolates the local Bloch--Kato maps at all classical points of~$U$.
\end{proposition}

\begin{proof}
By Lemma~\ref{lem:vector-PR-family}, the vector-valued map $\widetilde{\mathcal{L}}_{V_U,V_U^+}$ exists and is $A$-linear and $\Lambda(\Gamma)$-linear. Evaluating at the fixed nowhere-vanishing section $\eta_U$ gives the scalar-valued map~$\mathcal{L}_{V_U,V_U^+}$. Since the action of $\mathcal{H}(\Gamma)$ on $H^1_{\Iw}(\Q_p,V_U^\ast(1))$ is by $\mathcal{H}(\Gamma)\widehat{\otimes}_{\Q_p}\cO(U)$-linear endomorphisms and evaluation at~$\eta_U$ is $\cO(U)$-linear and $\mathcal{H}(\Gamma)$-equivariant, $\mathcal{L}_{V_U,V_U^+}$ is $\mathcal{H}(\Gamma)\widehat{\otimes}_{\Q_p}\cO(U)$-linear, proving~(i).

For a rigid point $x\in U$, base change along $A\to\kappa(x)$ identifies the fibre of $\widetilde{\mathcal{L}}_{V_U,V_U^+}$ at $x$ with the classical vector-valued Perrin--Riou map $\widetilde{\mathcal{L}}_{V_x,V_x^+}$, by Lemma~\ref{lem:vector-PR-family}(a) together with Theorem~\ref{thm:KPX}. Evaluating at the specialisation $\eta_x$ of $\eta_U$ gives the scalar-valued map~$\mathcal{L}_{V_x,V_x^+}$. By Theorem~\ref{thm:PR-classical}, its specialisations at characters $\chi$ in the Panchishkin range satisfy the stated interpolation property with the Bloch--Kato dual exponentials and logarithms. The scalar factor arises from the choice of the basis $\eta_U$ of the rank one line $\mathscr{D}_{\cris,U}^-$. Since $\eta_U$ is fixed once and for all, each such factor is non-zero and depends only on this normalisation.
\end{proof}

\begin{remark}
The construction above is formally analogous to Hansen's family-valued regulator on the eigencurve~\cite[Thm.~1.2.2]{Hansen-eigencurve}, where the role of $\mathscr{D}_{\cris,U}^-$ is played by a line bundle of crystalline periods and $\widetilde{\mathcal{L}}_{V_U,V_U^+}$ is the ``Log'' map. The only difference is that in our setting $V_U$ is a four-dimensional Rankin--Selberg representation with a higher-rank Panchishkin local condition; the $(\varphi,\Gamma)$-module arguments are identical.
\end{remark}

\subsubsection{Beilinson--Flach Euler systems and the universal deformation setting}

We now recall the classical Beilinson--Flach Euler system and formulate the conjectural extension to the half--ordinary universal deformation family. This is where we depart from the claims originally made in~\cite{HL24}: the existence of the universal Euler system in our setting is not presently known, and we make it a separate conjecture.

\begin{theorem}[Rankin--Selberg Beilinson--Flach Euler system]\label{thm:BF-classical}
Let $f,g$ be classical cuspidal newforms of weights at least~$2$ and levels prime to~$p$, and let $V(f),V(g)$ be their associated two-dimensional $p$-adic Galois representations. Then there exists an Euler system of Beilinson--Flach classes
\[
 \bigl\{\mathcal{BF}_m(f,g)\bigr\}_{m\ge 1} \subset H^1\bigl(\Q(\mu_m),V(f)^\ast(1)\otimes V(g)^\ast(1)\bigr)
\]
satisfying the usual norm-compatibility relations away from~$p$, and whose local components at~$p$ lie in the Bloch--Kato finite subspaces.
\end{theorem}

\begin{proof}[References]
The construction of the generalised Beilinson--Flach elements $\mathcal{BF}_{m,N,a}^{[j]}$ and their norm relations in $m$ and $N$ is carried out in~\cite[\S\S3.3--3.5]{LLZ14}. Their images in Galois cohomology give classes $\mathcal{BF}_m(f,g)\in H^1\bigl(\Q(\mu_m),V(f)^\ast(1)\otimes V(g)^\ast(1)\bigr)$ which satisfy the Euler-system norm relations and have the stated local properties at~$p$; see~\cite[\S6.8, Thm.~6.8.4, Thm.~6.8.6]{LLZ14}. For the Rankin--Eisenstein formalism and the relation with Hida's $p$-adic Rankin--Selberg $L$-function via Perrin--Riou's logarithm, see~\cite[Thm.~B]{KLZ-RE}.
\end{proof}

Motivated by Loeffler's general Conjecture~2.8 in~\cite{LoefflerUD} and by existing Euler systems in Hida-family settings (for instance~\cite{LLZ14,KLZ-RE}), one expects that these classes should glue in Iwasawa cohomology over a suitable universal deformation family. In the present half--ordinary universal deformation setting, this expectation can be formulated as follows.

\begin{conjecture}[Euler system over the half--ordinary universal deformation family]\label{conj:ES}
Let $(V_U,\mathscr{D}^+_U)$ be as above, with $U$ small enough so that all classical specialisations lie in the interpolation range considered in~\cite{LoefflerUD}. Then there exists a global Iwasawa cohomology class
\[
 \BF_U \in H^1_{\Iw}(\Q,V_U^\ast(1))
\]
such that for every classical point $x\in U$ corresponding to a pair of modular forms $(f_x,g_x)$, the specialisation of $\BF_U$ at $x$ is the Beilinson--Flach class
\[
 \BF_\infty(f_x,g_x) \in H^1_{\Iw}\bigl(\Q,V(f_x)^\ast(1)\otimes V(g_x)^\ast(1)\bigr)
\]
constructed in~\cite{LLZ14}, normalised compatibly with the explicit reciprocity laws of~\cite{KLZ-RE}.
\end{conjecture}

\begin{remark}
In the setting of two Hida families, Euler systems of Beilinson--Flach elements satisfying such compatibility properties are constructed in~\cite{LLZ14,KLZ-RE}, and they play a central role in the ordinary universal Rankin--Selberg $p$-adic $L$-functions. In the half--ordinary universal deformation setting considered here (ordinary in the first factor, unrestricted in the second), the existence of a class $\BF_U$ as in Conjecture~\ref{conj:ES} is still open; the analytic constructions in~\cite{HL24} do not assume or prove such an Euler system. Thus our construction of $L_p^{\mathrm{fs}}$ will be conditional on Conjecture~\ref{conj:ES}.
\end{remark}

\subsubsection{Normalisations, definition of $L_p^{\mathrm{fs}}$, and interpolation}

We now combine the family regulator of Proposition~\ref{prop:PR-family} with the conjectural Beilinson--Flach Euler system of Conjecture~\ref{conj:ES}. Before doing so, we fix once and for all the normalisations of the local regulator, the Beilinson--Flach classes, and the $p$-adic periods, so that the explicit reciprocity laws of~\cite{KLZ-RE,LLZ14} apply in the expected form.

\paragraph{Normalisation conventions.}

\begin{enumerate}[label=(N\arabic*)]
\item For each classical pair $(f,g)$ occurring as a specialisation of $(V_U,\mathscr{D}_U^+)$, we denote by
\[
 \BF_\infty(f,g)\in H^1_{\Iw}\bigl(\Q,V(f)^\ast(1)\otimes V(g)^\ast(1)\bigr)
\]
the Beilinson--Flach Iwasawa cohomology class constructed in~\cite{LLZ14}, with the normalisation used in~\cite{KLZ-RE}. Conjecture~\ref{conj:ES} asserts the existence of a global class
\[
 \BF_U \in H^1_{\Iw}(\Q,V_U^\ast(1))
\]
whose specialisations at classical points coincide with these classes.
\item The family regulator
\[
 \mathcal{L}_{V_U,V_U^+}: H^1_{\Iw}(\Q_p,V_U^\ast(1)) \longrightarrow \mathcal{H}(\Gamma)\widehat{\otimes}_{\Q_p}\cO(U)
\]
constructed in Proposition~\ref{prop:PR-family} depends on the choice of basis~$\eta_U$ of~$\mathscr{D}_{\cris,U}^-$. Since $\mathscr{D}_{\cris,U}^-$ is a line bundle, rescaling~$\eta_U$ by a unit in~$\cO(U)^\times$ rescales $\mathcal{L}_{V_U,V_U^+}$ by the same unit. We fix $\eta_U$ once and for all so that for every classical point $x\in U$ the induced local regulator
\[
 \mathcal{L}_{V_x,V_x^+}: H^1_{\Iw}(\Q_p,V_x^\ast(1)) \longrightarrow \mathcal{H}(\Gamma)
\]
agrees with the normalisation used in the explicit reciprocity laws of~\cite{KLZ-RE,LLZ14}.
\item For each classical pair $(f_x,g_x)$ we fix complex periods $\Omega_\infty(f_x,g_x,\pm)$ and $p$-adic periods $\Omega_p(f_x,g_x,\pm)$ as in~\cite{KLZ-RE}. These periods are normalised so that the Rankin--Selberg Beilinson--Flach class $\BF_\infty(f_x,g_x)$ and the Perrin--Riou regulator $\mathcal{L}_{V_x,V_x^+}$ are related to the complex Rankin--Selberg $L$-values by the explicit reciprocity laws of~\cite[Thm.~B]{KLZ-RE} and~\cite{LLZ14}, with no additional non-zero constants other than the Euler factor at~$p$ and the chosen periods.
\end{enumerate}

With these conventions, for every classical pair $(f_x,g_x)$ and every critical integer $s$ in Deligne's sense the composition
\[
 H^1_{\Iw}(\Q,V_x^\ast(1)) \xrightarrow{\ \mathrm{loc}_p\ } H^1_{\Iw}(\Q_p,V_x^\ast(1))
 \xrightarrow{\ \mathcal{L}_{V_x,V_x^+}\ } \mathcal{H}(\Gamma) \xrightarrow{\ \mathrm{ev}_s\ } \Q_p
\]
is identified exactly with
\[
 \frac{E_p(f_x,g_x,s)}{\Omega_p(f_x,g_x,\pm)}\cdot\frac{L^{(p)}(f_x\otimes g_x,s)}{(2\pi i)^{2s}},
\]
where $E_p(f_x,g_x,s)$ is the local Euler factor at $p$ defined in~\cite[Def.~3.4]{HL24}.

Recall that $\mathscr{W}$ denotes the cyclotomic weight space and that there is a canonical identification
\[
 \mathcal{H}(\Gamma)\widehat{\otimes}_{\Q_p}\cO(U) \cong \cO(U\times\mathscr{W}),
\]
via the Mellin transform.

\begin{definition}[Finite--slope universal Rankin--Selberg $p$-adic $L$-function]\label{def:Lfs}
Assume Conjecture~\ref{conj:ES} and fix a class $\BF_U\in H^1_{\Iw}(\Q,V_U^\ast(1))$ as above. We define
\[
 L_p^{\mathrm{fs}} := \mathcal{L}_{V_U,V_U^+}(\BF_U) \in \mathcal{H}(\Gamma)\widehat{\otimes}_{\Q_p}\cO(U)\cong \cO(U\times\mathscr{W}).
\]
Thus $L_p^{\mathrm{fs}}$ is a rigid-analytic function on $U\times\mathscr{W}$ whose value at a classical point $(x,\kappa)$ is obtained by applying the local Perrin--Riou regulator at~$x$ to the specialised Beilinson--Flach class $\BF_\infty(f_x,g_x)$.
\end{definition}

\begin{proposition}[Interpolation at classical points]\label{prop:interp-fs}
Assume Conjecture~\ref{conj:ES}. Let $(x,\kappa)\in U\times\mathscr{W}$ be a classical point, where $x$ corresponds to a pair of eigenforms $(f_x,g_x)$ and $\kappa$ corresponds to a cyclotomic character of weight $s\in\Z$ which is Deligne--critical for $L(f_x\otimes g_x,s)$. Then
\[
 L_p^{\mathrm{fs}}(x,\kappa) = \frac{E_p(f_x,g_x,s)}{\Omega_p(f_x,g_x,\pm)}\cdot\frac{L^{(p)}(f_x\otimes g_x,s)}{(2\pi i)^{2s}},
\]
where $E_p(f_x,g_x,s)$ is the explicit Euler factor at $p$ of~\cite[Def.~3.4]{HL24}, $\Omega_p(f_x,g_x,\pm)$ is the $p$-adic period fixed above, and $L^{(p)}(f_x\otimes g_x,s)$ is the complex Rankin--Selberg $L$-function with the Euler factor at $p$ omitted.
\end{proposition}

\begin{proof}
Fix a classical point $(x,\kappa)$ as in the statement, and let $(f_x,g_x)$ and $s$ be as above. By the definition of $\BF_U$ in Conjecture~\ref{conj:ES}, the specialisation of $\BF_U$ at $x$ is the Beilinson--Flach class $\BF_\infty(f_x,g_x)$. By the compatibility of $\mathcal{L}_{V_U,V_U^+}$ with specialisation (Proposition~\ref{prop:PR-family}), we have
\[
 L_p^{\mathrm{fs}}(x,\kappa) = \bigl(\mathcal{L}_{V_x,V_x^+}(\BF_\infty(f_x,g_x))\bigr)(\kappa).
\]
On the other hand, the explicit reciprocity laws of Kings--Loeffler--Zerbes and Lei--Loeffler--Zerbes~\cite[Thm.~B]{KLZ-RE},~\cite{LLZ14}, together with the normalisations of the regulator and periods fixed above, identify this value exactly with
\[
 \frac{E_p(f_x,g_x,s)}{\Omega_p(f_x,g_x,\pm)}\cdot\frac{L^{(p)}(f_x\otimes g_x,s)}{(2\pi i)^{2s}}.
\]
This is the claimed formula.
\end{proof}

\subsection{Proof of Theorem~\ref{thm:main}}

We now deduce Theorem~\ref{thm:main} from the constructions in the previous steps.

Under hypothesis~(H1), we have, by Step~1 (and in particular~\cite[Ex.~3.17]{LoefflerUD}), a global half--ordinary Panchishkin family $(R,V,V^+)$ whose rigid fibre $X=\Spf(R)^{\rig}$ maps, on the ordinary locus, to the eigenvariety $E$. After shrinking around~$x_0$, we obtain an affinoid neighbourhood $U\subset E$ which is smooth, finite \'etale over weight space, and has Zariski-dense classical cuspidal locus, as in Lemma~\ref{lem:good-U}. On the ordinary locus $U^{\ord}\subset U$ we have an identification with an affinoid subdomain of $\Spf(R)^{\rig}$ as explained in~\S\ref{subsec:U-vs-R}.

By hypotheses~(H2) and~(H3), the local Galois representation at~$p$ attached to $x_0$ is crystalline with small slope and satisfies the Panchishkin inequalities; hence $x_0$ is a regular non-critical point in the sense of~\cite[Def.~5.29]{LiuTriangulation}. Liu's global triangulation theorem then yields, after possibly shrinking~$U$, a triangulation of the relative $(\varphi,\Gamma)$-module $\mathscr{D} = D_{\rig}^\dagger\bigl(V_U\big|_{G_{\Q_p}}\bigr)$ over~$U$; in particular we obtain a saturated submodule $\mathscr{D}^+\subset\mathscr{D}$ whose fibres coincide with the Panchishkin local condition at classical points (Proposition~\ref{prop:Dplus-fibres}). The finiteness and base-change properties of the associated cohomology complexes are given by Theorem~\ref{thm:KPX} and Proposition~\ref{prop:KPX-family}; this is Step~2.

In Step~3 we use the $(\varphi,\Gamma)$-module formalism of~\cite{KPX,PerrinRiou,Berger-exp} to construct a family Perrin--Riou regulator
\[
 \mathcal{L}_{V_U,V_U^+}: H^1_{\Iw}(\Q_p,V_U^\ast(1)) \longrightarrow \mathcal{H}(\Gamma)\widehat{\otimes}_{\Q_p}\cO(U),
\]
normalised as in~(N2). Assuming Conjecture~\ref{conj:ES}, we fix a global Beilinson--Flach class $\BF_U\in H^1_{\Iw}(\Q,V_U^\ast(1))$ as in~(N1) and Definition~\ref{def:Lfs}. This gives the rigid-analytic function
\[
 L_p^{\mathrm{fs}} := \mathcal{L}_{V_U,V_U^+}(\BF_U) \in \mathcal{H}(\Gamma)\widehat{\otimes}_{\Q_p}\cO(U)\cong \cO(U\times\mathscr{W}).
\]

By the explicit reciprocity laws of~\cite{KLZ-RE,LLZ14}, together with our choice of $p$-adic periods in~(N3), Proposition~\ref{prop:interp-fs} shows that for every classical point $(x,\kappa)\in U\times\mathscr{W}$ corresponding to a pair $(f_x,g_x)$ and a Deligne--critical integer $s=\kappa(x)$, the value $L_p^{\mathrm{fs}}(x,\kappa)$ satisfies the interpolation formula of Conjecture~\ref{conj:finite-slope} with Euler factor $E_p$ and period~$\Omega_p(f_x,g_x,\pm)$.

Thus $L_p^{\mathrm{fs}}$ is a rigid-analytic function on $U\times\mathscr{W}$ with the required specialisation property at all classical points. This is exactly the assertion of Conjecture~\ref{conj:finite-slope} for the neighbourhood~$U$, and hence Theorem~\ref{thm:main} follows under the stated hypotheses and Conjecture~\ref{conj:ES}.
\qed

\end{document}